\documentclass{amsart}
\usepackage{amssymb}
\usepackage{bm}
\usepackage{braket}

\newtheorem{theorem}{Theorem}
\newtheorem{lemma}{Lemma}
\theoremstyle{definition}
\newtheorem{definition}{Definition}
\newtheorem{remark}{Remark}

\begin{document}

\title[On amicable tuples]
{On amicable tuples}
\author[Y. Suzuki]{Yuta Suzuki}
\date{}


\begin{abstract}
For an integer $k\ge2$, a tuple of $k$ positive integers
$(M_i)_{i=1}^{k}$
is called an \textit{amicable $k$-tuple} if the equation
\[
\sigma(M_1)=\cdots=\sigma(M_k)=M_1+\cdots+M_k
\]
holds.
This is a generalization of amicable pairs.
An \textit{amicable pair} is a pair of distinct positive integers each of which is the sum of the proper divisors of the other.
Gmelin~(1917) conjectured that there is no relatively prime amicable pairs
and Artjuhov~(1975) and Borho~(1974) proved that for any fixed positive integer $K$,
there are only finitely many relatively prime amicable pairs $(M,N)$ with $\omega(MN)=K$.
Recently, Pollack~(2015) obtained an upper bound
\[
MN<(2K)^{2^{K^2}}
\]
for such amicable pairs. In this paper, we improve this upper bound to
\[
MN<\frac{\pi^2}{6}2^{4^K-2\cdot 2^K}
\]
and generalize this bound to some class of general amicable tuples.
\end{abstract}
\maketitle
\vspace{-5mm}

%
%
\section{Introduction}

For an integer $k\ge2$, a tuple of $k$ positive integers
\[
(M_1,\,\ldots\,,M_k)
\]
is called an \textit{amicable $k$-tuple}~\cite{Dickson_triple} if the equation
\begin{equation}
\label{EQ:definition_amicable}
\sigma(M_1)=\cdots=\sigma(M_k)=M_1+\cdots+M_k
\end{equation}
holds, where $\sigma(n)$ is the usual divisor summatory function.
This is a generalization of amicable pairs.
An \textit{amicable pair} is a pair of distinct positive integers each of which is the sum of the proper divisors of the other.
For a pair of distinct positive integers~$(M,N)$,
this definition of amicable pair can be rephrased as
\[
\sigma(M)-M=N\quad\text{and}\quad\sigma(N)-N=M,
\]
which is equivalent to
\[
\sigma(M)=\sigma(N)=M+N
\]
so $(M,N)$ is an amicable pair if and only if $(M,N)$ is an amicable 2-tuple.

Although the history of amicable pairs can be traced back more than 2000 years ago
to Pythagoreans, their nature is still shrouded in mystery.
In 1917, Gmelin~\cite{Gmelin} noted that
there is no amicable pairs $(M,N)$ on the list at that time
for which $M$ and $N$ are relatively prime.
Based on this observation, he conjectured that there is no such relatively prime amicable pair.

As for this problem, Artjuhov~\cite{Artjuhov} and Borho~\cite{Borho} proved that
for any fixed integer $K\ge1$, there are only finitely many relatively prime amicable pairs $(M,N)$
with $\omega(MN)=K$ where $\omega(n)$ denote the number of distinct prime factors of $n$.
Recently, Pollack~\cite{Pollack_bound} obtained an explicit upper bound 
\begin{equation}
\label{EQ:Pollack}
MN<(2K)^{2^{K^2}}
\end{equation}
for relatively prime amicable pairs $(M,N)$, where $K=\omega(MN)$.
The main aim of this paper is to improve and generalize this result of Pollack.

Actually, not only for amicable pairs as in Gmelin's conjecture,
members of each amicable tuple seem to share relatively many common factors.
Although there are not so many examples of large amicable $k$-tuples for $k\ge4$,
it seems further that there is no amicable tuple whose greatest common divisor is 1.
For some technical and artificial reason,
we consider a much stronger condition on amicable tuples.
We define this condition not only for amicable tuples but for general tuples.
\begin{definition}
For an integer $k\ge2$, we say a tuple of distinct positive integers $(M_i)_{i=1}^k$ is \textit{anarchy} if
\[
(M_i,M_j\sigma(M_j))=1
\]
for all distinct $i,j$,
where $(a,b)$ denotes the greatest common divisor of $a$ and $b$.
\end{definition}
We also introduce a new class of tuples of integers,
which has been essentially introduced
by Kozek, Luca, Pollack and Pomerance~\cite{Harmonious}.
\begin{definition}
For an integer $k\ge2$, we say a tuple of positive integers $(M_i)_{i=1}^k$
is a \textit{harmonious $k$-tuple} if the equation
\begin{equation}
\label{EQ:fundamental_eq}
\frac{M_1}{\sigma(M_1)}+\cdots+\frac{M_k}{\sigma(M_k)}=1
\end{equation}
holds.
\end{definition}
In this paper, we generalize and improve Pollack's upper bound \eqref{EQ:Pollack} as follows.
\begin{theorem}
\label{Thm:main}
For any anarchy harmonious tuple $(M_i)_{i=1}^{k}$, we have
\[
M_1\cdots M_k<\frac{\pi^2}{6}\,2^{4^K-2\cdot2^K},
\]
where $K=\omega(M_1\cdots M_k)$.
\end{theorem}

Note that if $(M_i)_{i=1}^{k}$ is an amicable tuple, then
\[
1
=\frac{M_1}{M_1+\cdots+M_k}+\cdots+\frac{M_k}{M_1+\cdots+M_k}
=\frac{M_1}{\sigma(M_1)}+\cdots+\frac{M_k}{\sigma(M_k)}
\]
so every amicable tuple is harmonious.
Also, for any amicable tuple $(M_i)_{i=1}^{k}$,
\[
(M_i,M_j\sigma(M_j))=(M_i,M_j(M_1+\ldots+M_k))
\]
so an amicable tuple $(M_i)_{i=1}^{k}$ is anarchy
if and only if $M_i$ are pairwise coprime and
$M_1\cdots M_k$ is coprime to $M_1+\cdots+M_k$.
Moreover, an amicable pair $(M,N)$ is anarchy if and only if $M$ and $N$ are relatively prime since
\[
(M(M+N),N)=1\ \Longleftrightarrow\ (M,N)=1\ \Longleftrightarrow\ (M,N(M+N))=1.
\]
Thus Theorem \ref{Thm:main} leads the following corollaries.
\begin{theorem}
\label{Thm:amicable}
For any relatively prime amicable pair $(M,N)$, we have
\[
MN<\frac{\pi^2}{6}\,2^{4^K-2\cdot2^K},
\]
where $K=\omega(MN)$.
\end{theorem}

\begin{theorem}
\label{Thm:amicable_tuple}
For any amicable tuple $(M_i)_{i=1}^{k}$ satisfying
\begin{equation}
\label{EQ:tuple_cond}
M_1,\ldots,M_k\text{ are pairwise coprime and }\,
(M_1\cdots M_k,M_1+\cdots+M_k)=1,
\end{equation}
we have
\[
M_1\cdots M_k<\frac{\pi^2}{6}\,2^{4^K-2\cdot2^K},
\]
where $K=\omega(M_1\cdots M_k)$.
\end{theorem}
\noindent
Therefore, Theorem \ref{Thm:amicable} improves the upper bound \eqref{EQ:Pollack}.

Even without divisibility condition like relative primality or anarchy,
Borho~\cite{Borho2} proved an upper bound for amicable pairs
in terms of $\Omega(n)$, the number of prime factors of $n$ counted with multiplicity.
He also dealt with \textit{unitary amicable pairs}.
A unitary amicable pair is a pair of positive integers $(M,N)$ satisfying the equation
\[
\sigma^\ast(M)=\sigma^\ast(N)=M+N,\quad
\sigma^\ast(n):=\sum_{d\parallel n}d,
\]
where $d\parallel n$ means $d|n$ and $(d,n/d)=1$. Borho's upper bound is
\[
MN<L^{2^{L}}
\]
for any amicable or unitary amicable pairs $(M,N)$ with $L=\Omega(MN)$ for amicable pairs
and $L=\omega(M)+\omega(N)$ for unitary amicable pairs.
Indeed, we can see a prototype of our Lemma \ref{Lem:HB_ineq1} and \ref{Lem:HB_ineq2} on Diophantine equation
in Satz 1 of Borho~\cite{Borho2}. By introducing our lemmas in Borho's argument,
we can improve and generalize Borho's theorem.
As Kozek, Luca, Pollack and Pomerance~\cite{Harmonious} remarked,
what we can deal with are not only amicable tuples but harmonious tuples.
Thus we first introduce the unitary analogue of harmonious tuples.
\begin{definition}
For an integer $k\ge2$, we say a tuple of positive integers $(M_i)_{i=1}^k$
is a \textit{unitary harmonious tuple} if the equation
\begin{equation}
\label{EQ:fundamental_eq_unitary}
\frac{M_1}{\sigma^\ast(M_1)}+\cdots+\frac{M_k}{\sigma^\ast(M_k)}=1
\end{equation}
holds.
\end{definition}

Then our theorem of the Borho-type is the following.
\begin{theorem}
\label{Thm:Omega}
For any harmonious tuple $(M_i)_{i=1}^{k}$, we have
\[
M_1\cdots M_k\le k^{-k}(2^{2^{L}}-2^{2^{L-1}}),
\]
where $L=\Omega(M_1\cdots M_k)$. For any unitary harmonious tuple $(M_i)_{i=1}^{k}$, we again have
the same upper bound but $L$ is replaced by $L=\omega(M_1)+\cdots+\omega(M_k)$.
\end{theorem}
We prove Theorem~\ref{Thm:Omega} in Section~\ref{Section:Omega}.

The above problems on the upper bound of harmonious or amicable tuples
are analogues of a similar problem in the context of odd perfect numbers,
which has been studied since a long time ago.
A positive integer $N$ is called a \textit{perfect number} if its sum of all proper divisors
is equal to $N$ itself, i.e.~if $\sigma(N)=2N$.
It is also long-standing mystery whether or not there is an odd perfect number.
The finiteness theorem like the Artjuhov--Borho theorem
was proved for odd perfect numbers by Dickson~\cite{Dickson_bound} in 1913.
However, it took more than 60 years to get the explicit upper bound.
The first explicit upper bound for odd perfect numbers was achieved
by Pomerance~\cite{Pomerance_bound} in 1977. Pomerance obtained an upper bound
\[
N<(4K)^{(4K)^{2^{K^2}}}
\]
for an odd perfect number with $K=\omega(N)$.
Note that by modifying the method of Pomerance slightly,
we may improve this upper bound to
\[
N<(4K)^{2^{K^2}}
\]
as remarked in Lemma 2 and Remark of \cite{Pollack_bound}.
Further improvement on this upper bound
was given by Heath-Brown~\cite{HB_OPN} by using a new method.
Heath-Brown's upper bound is
\[
N<4^{4^{K}}.
\]
Heath-Brown's method has been further developed by several authors.
The list of the world records of upper bounds based on Heath-Bronws's method is here:
\[
\begin{array}{ll}
C^{4^{K}}&(\text{Cook~\cite{Cook}}),\quad(C=(195)^{1/7}=2.123\ldots\,),\\
2^{4^{K}}&(\text{Nielsen~\cite{Nielsen}}),\\
2^{4^{K}-2^K}&(\text{Chen and Tang~\cite{Chen_Tang}}),\\
2^{(2^K-1)^2}&(\text{Nielsen~\cite{Nielsen_New}}),
\end{array}
\]
where the last result of Nielsen is the current best result.

The method of Pollack~\cite{Pollack_bound} was mainly based on the method of Pomerance.
And Pollack~\cite[p.~38]{Pollack_bound} suggested that it would be interesting to find
whether the method of Heath-Brown is available in the context of amicable pairs.
Indeed, the method we use in this paper mainly follows a version of Heath-Brown's method
given by Nielsen~\cite{Nielsen_New} and Theorem \ref{Thm:amicable}
corresponds to Nielsen's latest bound
on odd perfect numbers. Thus, this paper gives one possible answer to Pollack's suggestion.

Also, Pollack~\cite[p.~680]{Pollack_kin} asked to find a suitable condition
to obtain a finiteness theorem for amicable tuples or sociable numbers.
Our condition \eqref{EQ:tuple_cond} in Theorem~\ref{Thm:amicable_tuple}
can be, though it seems too strong, one of partial answers to his question.
However, the current author have no idea for the same problem with sociable numbers.

It would be interesting to note that there are many examples of harmonious pairs
which are relatively prime but not anarchy.
In order to list up such pairs, we used a \texttt{C} program,
which is based on a program provided by Yuki Yoshida~\cite{yos}.
By using this program, we can list up all $2566$ relatively prime harmonious pairs
among all $49929$ harmonious pairs $(M,N)$ up to $10^8$ in the sense $M\le N\le10^8$,
and we find none of these examples is anarchy.
For interested readers, we listed up all $30$ relatively prime harmonious pairs up to $10^5$ in Table~\ref{Table:PHP1}
and listed the number of harmonious and relatively prime harmonious pairs in several ranges in Table~\ref{Table:PHP2}
This numerical search shows that Theorem~\ref{Thm:main} captures more pairs in its scope than Gmelin's conjecture.
Therefore, it is natural to ask:~are there anarchy harmonious tuples?
Surprisingly, by continuing the numerical search, we find an \textit{anarchy in the harmony},
i.e.~an anarchy harmonious pair
\[
(M,N)=(64,173369889),
\]
which is the only anarchy harmonious pair with $M\le N\le10^9$.
Note that
\[
64=2^6,\quad
173369889=3^4\times7^2\times11^2\times19^2,\quad
\]
\[
\sigma(64)=127,\quad
\sigma(173369889)=3^2\times7\times11^2\times19^2\times127.
\]
The next natural problem may be: how many anarchy harmonious pairs are there?
Also, the above observations indicates some possibility to improve Theorem \ref{Thm:main}
by introducing a new condition stronger than anarchy.

\begin{table}[htb]
\renewcommand{\arraystretch}{1.06}
\centering
\caption{All coprime harmonious pairs $(M,N)$ with $M\le N\le 10^5$}
\label{Table:PHP1}
\begin{tabular}{|cc|cc||cc|}\hline
$M$ & $N$ & \multicolumn{2}{c||}{Factorization of $M$ and $N$} & $(M,\sigma(N))$ & $(\sigma(M),N)$\\ \hline \hline
$135$ & $3472$ & $3^3 \times 5$ & $2^4 \times 7 \times 31$ & $1$ & $16$ \\
$135$ & $56896$ & $3^3 \times 5$ & $2^6 \times 7 \times 127$ & $1$ & $16$ \\
$285$ & $45136$ & $3 \times 5 \times 19$ & $2^4 \times 7 \times 13 \times 31$ & $1$ & $16$ \\
$315$ & $51088$ & $3^2 \times 5 \times 7$ & $2^4 \times 31 \times 103$ & $1$ & $16$ \\
$345$ & $38192$ & $3 \times 5 \times 23$ & $2^4 \times 7 \times 11 \times 31$ & $3$ & $16$ \\
$868$ & $1485$ & $2^2 \times 7 \times 31$ & $3^3 \times 5 \times 11$ & $4$ & $1$ \\
$1204$ & $4455$ & $2^2 \times 7 \times 43$ & $3^4 \times 5 \times 11$ & $4$ & $11$ \\
$1683$ & $3500$ & $3^2 \times 11 \times 17$ & $2^2 \times 5^3 \times 7$ & $3$ & $4$ \\
$1683$ & $62000$ & $3^2 \times 11 \times 17$ & $2^4 \times 5^3 \times 31$ & $3$ & $8$ \\
$2324$ & $9945$ & $2^2 \times 7 \times 83$ & $3^2 \times 5 \times 13 \times 17$ & $28$ & $3$ \\
$3556$ & $63855$ & $2^2 \times 7 \times 127$ & $3^3 \times 5 \times 11 \times 43$ & $4$ & $1$ \\
$4455$ & $21328$ & $3^4 \times 5 \times 11$ & $2^4 \times 31 \times 43$ & $11$ & $8$ \\
$4845$ & $7084$ & $3 \times 5 \times 17 \times 19$ & $2^2 \times 7 \times 11 \times 23$ & $3$ & $4$ \\
$5049$ & $65968$ & $3^3 \times 11 \times 17$ & $2^4 \times 7 \times 19 \times 31$ & $1$ & $16$ \\
$6244$ & $43875$ & $2^2 \times 7 \times 223$ & $3^3 \times 5^3 \times 13$ & $28$ & $1$ \\
$6244$ & $90675$ & $2^2 \times 7 \times 223$ & $3^2 \times 5^2 \times 13 \times 31$ & $28$ & $1$ \\
$6675$ & $33488$ & $3 \times 5^2 \times 89$ & $2^4 \times 7 \times 13 \times 23$ & $3$ & $8$ \\
$7155$ & $13244$ & $3^3 \times 5 \times 53$ & $2^2 \times 7 \times 11 \times 43$ & $3$ & $4$ \\
$9945$ & $41168$ & $3^2 \times 5 \times 13 \times 17$ & $2^4 \times 31 \times 83$ & $3$ & $8$ \\
$12124$ & $84825$ & $2^2 \times 7 \times 433$ & $3^2 \times 5^2 \times 13 \times 29$ & $28$ & $1$ \\
$13275$ & $81424$ & $3^2 \times 5^2 \times 59$ & $2^4 \times 7 \times 727$ & $1$ & $4$ \\
$13965$ & $23312$ & $3 \times 5 \times 7^2 \times 19$ & $2^4 \times 31 \times 47$ & $3$ & $16$ \\
$24327$ & $75460$ & $3^3 \times 17 \times 53$ & $2^2 \times 5 \times 7^3 \times 11$ & $9$ & $20$ \\
$31724$ & $61335$ & $2^2 \times 7 \times 11 \times 103$ & $3^2 \times 5 \times 29 \times 47$ & $4$ & $3$ \\
$32835$ & $92456$ & $3 \times 5 \times 11 \times 199$ & $2^3 \times 7 \times 13 \times 127$ & $15$ & $8$ \\
$34485$ & $37492$ & $3 \times 5 \times 11^2 \times 19$ & $2^2 \times 7 \times 13 \times 103$ & $1$ & $28$ \\
$52700$ & $68211$ & $2^2 \times 5^2 \times 17 \times 31$ & $3^2 \times 11 \times 13 \times 53$ & $4$ & $9$ \\
$55341$ & $58900$ & $3^2 \times 11 \times 13 \times 43$ & $2^2 \times 5^2 \times 19 \times 31$ & $1$ & $4$ \\
$60515$ & $78864$ & $5 \times 7^2 \times 13 \times 19$ & $2^4 \times 3 \times 31 \times 53$ & $1$ & $48$ \\
$62992$ & $63855$ & $2^4 \times 31 \times 127$ & $3^3 \times 5 \times 11 \times 43$ & $16$ & $1$ \\
\hline
\end{tabular}
\end{table}

\begin{table}[htb]
\renewcommand{\arraystretch}{1.06}
\centering
\setlength{\tabcolsep}{4pt}
\caption{Number of harmonious pairs $(M,N)$ with $M\le N\le 10^k$}
\label{Table:PHP2}
\begin{tabular}{|c||ccccccccc|}\hline
 & $10$ & $\phantom{^2}10^2$ & $\phantom{^3}10^3$ & $\phantom{^4}10^4$ &
 $\phantom{^5}10^5$ & $\phantom{^6}10^6$ & $\phantom{^7}10^7$ & $\phantom{^8}10^8$ & $\phantom{^9}10^9$\\ \hline \hline
Harmonious & $1$ & $10$ & $55$ & $252$ & $983$ & $3666$ & $13602$ & $49929$ & $176453$\\
Coprime harmonious & $0$ & $0$ & $0$ & $6$ & $30$ & $133$ & $631$ & $2566$ & $10013$\\
\hline
\end{tabular}
\end{table}

%
%
\section{Notation}
We denote the greatest common divisor of positive integers $a$ and $b$ by $(a,b)$,
which we may distinguish from the notation for a pair of integers $(M,N)$ by the context.
For positive integers $d$ and $n$, we write $d\parallel n$ if $d\mid n$ and $(d,n/d)=1$.

For any finite set $\mathcal{S}$ of integers, we let
\[
\Pi(\mathcal{S})=\prod_{m\in \mathcal{S}}m,\quad
\Phi(\mathcal{S})=\prod_{m\in \mathcal{S}}(m-1),\quad
\Psi(\mathcal{S})=\Pi(\mathcal{S})\Phi(\mathcal{S}).
\]
Following the notation of Nielsen \cite{Nielsen_New}, we let
\[
F_r(x)=x^{2^{r}}-x^{2^{r-1}}
\]
for an integer $r\ge1$ and a real number $x\ge1$ and we let $F_0(x)=x-1$ for the case $r=0$ and $x\ge1$.
Actually, we will not use the full power of this function $F_r(x)$ for the proof of Theorem \ref{Thm:main},
but we introduce them for trying to give a better bound in lemmas on Diophantine inequalities.
We prepare two lemmas on $F_r(x)$.
\begin{lemma}
\label{Lem:F_increasing}
For any integer $r\ge0$, $F_r(x)$ is increasing as a function of $x\ge1$.
\end{lemma}
\begin{proof}
This is obvious for $r=0$ and also obvious for $r\ge1$ from the factorization
\begin{equation}
\label{EQ:F_factorization}
F_r(x)=x^{2^{r-1}}(x^{2^{r-1}}-1)
\end{equation}
since both of $x^{2^{r-1}}$ and $(x^{2^{r-1}}-1)$ are increasing and non-negative for $x\ge1$.
\end{proof}

\begin{lemma}
\label{Lem:F_scaling}
For any integer $r\ge1$ and real numbers $\alpha,x\ge1$, we have
\[
F_r(x)\le\alpha^{-2^{r-1}}F_r(\alpha x)\le\alpha^{-1}F_r(\alpha x).
\]
\end{lemma}
\begin{proof}
By \eqref{EQ:F_factorization}, we have
\[
\alpha^{-2^{r-1}}F_r(\alpha x)
=
x^{2^{r-1}}\left((\alpha x)^{2^{r-1}}-1\right)
\ge
x^{2^{r-1}}\left(x^{2^{r-1}}-1\right)
=
F_r(x).
\]
This completes the proof.
\end{proof}

%
%
\section{Lemmas on Diophantine inequalities}
\label{Section:Diophantine_inequality}
In this section,
we prove variants of Heath-Brown's lemma~\cite[Lemma 1]{HB_OPN}
on Diophantine inequalities related to the equation \eqref{EQ:fundamental_eq}. We need to introduce some modification
suitable for applications to amicable tuples.
Our proof of Theorem \ref{Thm:main} heavily relies on the equation \eqref{EQ:fundamental_eq},
or its generalization
\begin{equation}
\label{EQ:fundamental_eq_amicable}
\frac{b_1}{a_1}\frac{M_1}{\sigma(M_1)}+\cdots+\frac{b_k}{a_k}\frac{M_k}{\sigma(M_k)}=1,
\end{equation}
where $a_i,b_i\ge1$ are integers.
This equation is not so flexible as the equation
\begin{equation}
\label{EQ:fundamental_eq_OPN}
\frac{\sigma(M)}{M}=\frac{a}{b},
\end{equation}
which is used in the context of perfect numbers.

Actually, in the induction steps of Heath-Brown's method,
there are two point to use such Diophantine equation or corresponding inequalities.
For the first point, we use Diophantine inequality with its original form.
On the other hand, in the second point,
we need to take the ``reciprocal'' of the same Diophantine inequality.
For odd perfect numbers,
we can take the reciprocal of \eqref{EQ:fundamental_eq_OPN}
without any big change of its shape.
However, for amicable pairs, we need to take the reciprocal of each terms
in \eqref{EQ:fundamental_eq_amicable}, which transforms the equation into slightly different shape.
Thus, we prepare two different lemmas.

We start with Lemma 2 of Cook~\cite{Cook} in its refined form.
This refinement was given by Goto~\cite[Lemma 2.4]{Goto}.
Nielsen~\cite[Lemma 1.2]{Nielsen_New} also proved this refinement
in even stronger form, which allows us to have some equalities between $x_i$.
We also need a variant of Cook's lemma given by Goto~\cite[Lemma 2.5]{Goto}.
For completeness, we give a proof of these lemmas following the argument of Nielsen~\cite{Nielsen_New}.
%
%
\begin{lemma}
\label{Lem:pre_Cook}
For real numbers $0<x_1\le x_2$ and $0<\alpha<1$, we have
\[
\left(1-\frac{1}{x_1}\right)\left(1-\frac{1}{x_2}\right)
>\left(1-\frac{1}{x_1\alpha}\right)\left(1-\frac{1}{x_2\alpha^{-1}}\right)
\]
and
\[
\left(1+\frac{1}{x_1}\right)\left(1+\frac{1}{x_2}\right)
<\left(1+\frac{1}{x_1\alpha}\right)\left(1+\frac{1}{x_2\alpha^{-1}}\right).
\]
\end{lemma}
\begin{proof}
By expanding both sides of the inequalities,
we find that it suffices to prove
\[
\frac{1}{x_1}+\frac{1}{x_2}<\frac{1}{x_1\alpha}+\frac{1}{x_2\alpha^{-1}}.
\]
This is equivalent to
\[
\left(\alpha-\frac{x_2}{x_1}\right)\left(\alpha-1\right)>0.
\]
Since $\alpha<1\le x_2/x_1$, the last inequality holds.
This completes the proof.
\end{proof}

%
%
\begin{lemma}
\label{Lem:Cook}
Let
\begin{equation}
\label{EQ:xy_Cook_condition1}
1<x_1\le x_2\le\cdots\le x_k,\quad
1<y_1\le y_2\le\cdots\le y_k
\end{equation}
be sequences of real numbers satisfying
\begin{equation}
\label{EQ:xy_Cook_condition2}
\prod_{i=1}^{s}x_i\le\prod_{i=1}^{s}y_i
\end{equation}
for every $s$ with $1\le s\le k$.
Then we have
\[
\prod_{i=1}^{k}\left(1-\frac{1}{x_i}\right)\le\prod_{i=1}^{k}\left(1-\frac{1}{y_i}\right),\quad
\prod_{i=1}^{k}\left(1+\frac{1}{x_i}\right)\ge\prod_{i=1}^{k}\left(1+\frac{1}{y_i}\right),
\]
where each of two equalities holds if and only if $x_i=y_i$ for every $i\ge1$.
\end{lemma}
\begin{proof}
We first fix the tuple $\mathbf{x}=(x_i)$.
Let us identify the tuple $\mathbf{y}=(y_i)$ with a point in the Euclidean space $\mathbb{R}^{k}$
and let
\[
\mathcal{R}=\Set{\mathbf{y}\in\mathbb{R}^k| y_1\le\cdots\le y_k\ \text{and}\ 
\prod_{i=1}^{s}x_i\le\prod_{i=1}^{s}y_i\ \text{for all $1\le s\le k$}},
\]
\[
G(\mathbf{y})=\prod_{i=1}^{k}\left(1-\frac{1}{y_i}\right),\quad
H(\mathbf{y})=\prod_{i=1}^{k}\left(1+\frac{1}{y_i}\right)
\]
Note that the condition $y_1>1$ in \eqref{EQ:xy_Cook_condition1} is assured by
the case $s=1$ of \eqref{EQ:xy_Cook_condition2} since $y_1\ge x_1>1$.
Thus what we have to prove is that
the minimum value of $G(\mathbf{y})$ and the maximum value of $H(\mathbf{y})$
for $\mathbf{y}\in\mathcal{R}$ is taken only at $\mathbf{y}=\mathbf{x}$.

Note that $G(\mathbf{y})$ is increasing in every variable $y_i$ and $H(\mathbf{y})$ is decreasing in every variable $y_i$.
and that if $\mathbf{y}\in\mathcal{R}$, then
\[
\left(\min\left(y_1,\prod_{i=1}^{k}x_i\right),\ldots,\min\left(y_k,\prod_{i=1}^{k}x_i\right)\right)\in\mathcal{R}.
\]
Thus, the minimum value of $G(\mathbf{y})$ and the maximum value of $H(\mathbf{y})$ for $\mathbf{y}\in\mathcal{R}$
exists and is taken in the closed set
$\mathcal{R}\cap[1,\prod_{i=1}^{k}x_i]^k$.

Take $\mathbf{y}\in\mathbf{R}$ with $\mathbf{y}\neq\mathbf{x}$ arbitrarily.
Since we proved the existence of the minimum and maximum values of $G(\mathbf{y})$ and $H(\mathbf{y})$,
it suffices to prove that we can modify $\mathbf{y}$ to $\tilde{\mathbf{y}}\in\mathcal{R}$
such that $G(\mathbf{y})>G(\tilde{\mathbf{y}})$ and $H(\mathbf{y})<H(\tilde{\mathbf{y}})$.
Take the smallest index $t$ with $x_t\neq y_t$ and $1\le t\le k$. Then we have
\begin{equation}
\label{EQ:xy_equal}
x_i=y_i\ \ \text{for all $1\le i<t$}
\end{equation}
so that
\begin{equation}
\label{EQ:xy_product_equal}
\prod_{i=1}^{t-1}x_i=\prod_{i=1}^{t-1}y_i.
\end{equation}
By \eqref{EQ:xy_Cook_condition2}, \eqref{EQ:xy_product_equal} and $x_t\neq y_t$, we have
\begin{equation}
\label{EQ:xy_phase_transition}
x_t<y_t,\quad\text{so}\quad\prod_{i=1}^{t}x_i<\prod_{i=1}^{t}y_i.
\end{equation}
If $t$ is the last index, i.e.~$t=k$,
then by \eqref{EQ:xy_equal} and \eqref{EQ:xy_phase_transition} we have
\begin{align*}
G(\mathbf{y})
=\prod_{i=1}^{k-1}\left(1-\frac{1}{x_i}\right)\left(1-\frac{1}{y_k}\right)
>\prod_{i=1}^{k-1}\left(1-\frac{1}{x_i}\right)\left(1-\frac{1}{x_k}\right)
=G(\mathbf{x}),
\end{align*}
\begin{align*}
H(\mathbf{y})
=\prod_{i=1}^{k-1}\left(1+\frac{1}{x_i}\right)\left(1+\frac{1}{y_k}\right)
<\prod_{i=1}^{k-1}\left(1+\frac{1}{x_i}\right)\left(1+\frac{1}{x_k}\right)
=H(\mathbf{x}).
\end{align*}
Thus $\tilde{\mathbf{y}}=\mathbf{x}$ satisfies the conditions.
Hence we may assume $1\le t\le k-1$.

We next take the smallest index $v$ with $y_v=y_{t+1}$ and $t<v\le k$.
Then we have
\begin{equation}
\label{EQ:v_cond}
y_{t+1}=y_{t+2}=\cdots=y_{v}<y_{v+1},
\end{equation}
where we use a convention $y_{k+1}=2y_k$.
Now we prove
\begin{equation}
\label{EQ:critical_range_ineq}
\prod_{i=1}^{s}x_i<\prod_{i=1}^{s}y_i\ \ \text{for all $t\le s<v$}.
\end{equation}
Assume to the contrary that there is $s$ with $t\le s<v$ and
\[
\prod_{i=1}^{s}x_i\ge\prod_{i=1}^{s}y_i.
\]
By \eqref{EQ:xy_Cook_condition2}, we find that
\begin{equation}
\label{EQ:s_product_equal}
\prod_{i=1}^{s}x_i=\prod_{i=1}^{s}y_i.
\end{equation}
Thus, by the second inequality in \eqref{EQ:xy_phase_transition},
\[
\prod_{i=t+1}^{s}x_i
=
\left(\prod_{i=1}^{t}x_i\right)^{-1}\prod_{i=1}^{s}x_i
>
\left(\prod_{i=1}^{t}y_i\right)^{-1}\prod_{i=1}^{s}y_i
=
\prod_{i=t+1}^{s}y_i.
\]
Combined with \eqref{EQ:xy_Cook_condition1} and \eqref{EQ:v_cond}, this gives
\[
x_{s+1}\ge x_s
\ge\left(\prod_{i=t+1}^{s}x_i\right)^{1/(s-t)}
>\left(\prod_{i=t+1}^{s}y_i\right)^{1/(s-t)}=y_{s}=y_{s+1}
\]
since $s+1\le v$. By multiplying both sides by \eqref{EQ:s_product_equal}, we obtain
\[
\prod_{i=1}^{s+1}x_i>\prod_{i=1}^{s+1}y_i,
\]
which contradicts to the assumption \eqref{EQ:xy_Cook_condition2}.
Thus we obtain \eqref{EQ:critical_range_ineq}.

By $x_t<y_t$ , $y_v<y_{v+1}$ and \eqref{EQ:critical_range_ineq},
we find that
\[
0\le\max\left(x_t/y_t,\ y_{v}/y_{v+1},\ \prod_{i=1}^{t}x_i/y_i,\ \ldots,\ \prod_{i=1}^{v-1}x_i/y_i\right)<1
\]
so we can take a positive real number $\alpha$ with
\begin{equation}
\label{EQ:alpha_choice}
\max\left(x_t/y_t,\ y_{v}/y_{v+1},\ \prod_{i=1}^{t}x_i/y_i,\ \ldots,\ \prod_{i=1}^{v-1}x_i/y_i\right)<\alpha<1.
\end{equation}
We now define $\tilde{\mathbf{y}}=(\tilde{y}_i)\in\mathbb{R}^k$ by
\begin{equation}
\label{EQ:tilde_choice}
\tilde{y}_i=
\left\{
\begin{array}{ll}
y_i&(\text{for $i\neq t,v$}),\\
y_i\alpha&(\text{for $i=t$}),\\
y_i\alpha^{-1}&(\text{for $i=v$}),\\
\end{array}
\right.
\end{equation}
and check that this $\tilde{\mathbf{y}}$ satisfies the desired conditions.

First, we check
\begin{equation}
\label{EQ:tilde_Cook_condition1}
\tilde{y}_1\le\cdots\le\tilde{y}_k,\quad\text{i.e.}\quad
\tilde{y}_i\le\tilde{y}_{i+1}\ \ \text{for all $1\le i<k$}.
\end{equation}
This is obvious for $i\not\in\{t-1,t,v-1,v\}$ since in this case, $\tilde{y}_i$ and $\tilde{y}_{i+1}$
coincide with $y_i$ and $y_{i+1}$ respectively.
For the case $i\in\{t-1,t,v-1,v\}$ and $1\le i<k$, we use \eqref{EQ:alpha_choice} to check
\[
\tilde{y}_{t-1}=y_{t-1}=x_{t-1}\le x_t=y_t\cdot(x_t/y_t)<y_t\alpha=\tilde{y}_t,
\]
\[
\tilde{y}_{t}=y_t\alpha<y_t\le y_{t+1}\le\tilde{y}_{t+1},\quad
\tilde{y}_{v-1}\le y_{v-1}\le y_v<y_v\alpha^{-1}=\tilde{y}_v,
\]
\[
\tilde{y}_v=y_v\alpha^{-1}<y_v\cdot(y_v/y_{v+1})^{-1}=y_{v+1}=\tilde{y}_{v+1}.
\]
Thus the condition \eqref{EQ:tilde_Cook_condition1} holds.

Second, we check
\begin{equation}
\label{EQ:tilde_Cook_condition2}
\prod_{i=1}^{s}x_i\le\prod_{i=1}^{s}\tilde{y}_i\ \ \text{for all $1\le s\le k$}.
\end{equation}
This is obvious for $1\le s<t$ and $v\le s\le k$ since in these cases, we have
\[
\prod_{i=1}^{s}\tilde{y}_i=\prod_{i=1}^{s}y_i
\]
by our choice \eqref{EQ:tilde_choice}.
For $t\le s<v$, we see that
\[
\prod_{i=1}^{s}x_i
=\left(\prod_{i=1}^{s}x_i/y_i\right)\left(\prod_{i=1}^{s}y_i\right)
<\alpha\prod_{i=1}^{s}y_i
=\prod_{i=1}^{s}\tilde{y}_i
\]
by \eqref{EQ:alpha_choice}.
Thus the condition \eqref{EQ:tilde_Cook_condition2} also holds, i.e. $\tilde{\mathbf{y}}\in\mathcal{R}$.

Finally, we check that $G(\mathbf{y})>G(\tilde{\mathbf{y}})$ and $H(\mathbf{y})<H(\tilde{\mathbf{y}})$.
Since $0<\alpha<1$ and $y_t\le y_{t+1}\le y_v$,
by recalling \eqref{EQ:tilde_choice},
we can apply Lemma \ref{Lem:pre_Cook} to obtain
\begin{align*}
G(\mathbf{y})
&=
\prod_{\substack{i=1\\i\neq t,v}}^{k}\left(1-\frac{1}{\tilde{y}_i}\right)
\left(1-\frac{1}{y_t}\right)\left(1-\frac{1}{y_v}\right)\\
&>\prod_{\substack{i=1\\i\neq t,v}}^{k}\left(1-\frac{1}{\tilde{y}_i}\right)
\left(1-\frac{1}{y_t\alpha}\right)\left(1-\frac{1}{y_v\alpha^{-1}}\right)
=
G(\tilde{\mathbf{y}}),
\end{align*}
\begin{align*}
H(\mathbf{y})
&=
\prod_{\substack{i=1\\i\neq t,v}}^{k}\left(1+\frac{1}{\tilde{y}_i}\right)
\left(1+\frac{1}{y_t}\right)\left(1+\frac{1}{y_v}\right)\\
&<\prod_{\substack{i=1\\i\neq t,v}}^{k}\left(1+\frac{1}{\tilde{y}_i}\right)
\left(1+\frac{1}{y_t\alpha}\right)\left(1+\frac{1}{y_v\alpha^{-1}}\right)
=
H(\tilde{\mathbf{y}}).
\end{align*}
Thus our $\tilde{\mathbf{y}}$ satisfies the desired conditions.
This completes the proof.
\end{proof}

%
%
We next prove the first lemma on Dipophantine inequality.
The Diophantine inequality in the next lemma
seems to be a natural generalization of the Diophantine inequality (2) of \cite{HB_OPN}
to the linear form with several summands.
\begin{lemma}
\label{Lem:HB_ineq1}
Let $k\ge1$ be an integer.
For $R\ge1$, consider a sequence of integers
\[
\mathcal{M}=(m_j)_{j=1}^{R},\quad1<m_1\le m_2\le\cdots\le m_R,
\]
a decomposition of the index set
\[
J=\{1,\,\ldots,\,R\},\quad
J=\bigcup_{i=1}^{k}J_i,\quad
J_1,\,\ldots,\,J_k\colon\text{disjoint}
\]
and a tuple of integers
\[a_1,\,\ldots,\,a_k,\,b_1,\,\ldots,\,b_k\ge1\]
satisfying $a_i\ge b_i$ for all $i$.
If a pair of inequalities
\begin{equation}
\label{EQ:HB_assump_A1}
\sum_{i=1}^{k}\frac{b_i}{a_i}\prod_{\substack{j=1\\j\in J_i}}^{R}\left(1-\frac{1}{m_j}\right)\le1
\end{equation}
\begin{equation}
\label{EQ:HB_assump_B1}
\sum_{i=1}^{k}\frac{b_i}{a_i}\prod_{\substack{j=1\\j\in J_i\\}}^{R-1}\left(1-\frac{1}{m_j}\right)>1
\end{equation}
holds, then we have
\begin{equation}
\label{EQ:M_bound1}
a\prod_{j=1}^{R}m_j\le F_{R}(a+1),
\end{equation}
where $a=a_1\cdots a_k$.
\end{lemma}
\begin{proof}
We first give a preliminary remark.
By \eqref{EQ:HB_assump_B1},
we always have
\begin{equation}
\label{EQ:abcd_ineq_pre1}
\sum_{i=1}^{k}\frac{b_i}{a_i}>1
\end{equation}
since each of $\Pi(1-1/m)$ is $\le1$ even for the case they are empty products.
Since the left-hand side of \eqref{EQ:abcd_ineq_pre1} is a rational fraction
with denominator $a$,
\begin{equation}
\label{EQ:abcd_ineq1}
\sum_{i=1}^{k}\frac{b_i}{a_i}\ge\frac{a+1}{a}.
\end{equation}
We use this inequality several times below.

We use induction on $R$.
If $R=1$, by symmetry we may assume without loss of generality that
$J_1=\cdots=J_{k-1}=\emptyset$ and $J_{k}=\{1\}$.
Then \eqref{EQ:HB_assump_A1} implies
\[
\sum_{i=1}^{k}\frac{b_i}{a_i}-\frac{b_k}{a_k}\cdot\frac{1}{m_1}\le1.
\]
By \eqref{EQ:abcd_ineq1}, we find that
\[
\frac{b_k}{a_k}\cdot\frac{1}{m_1}\ge\sum_{i=1}^{k}\frac{b_i}{a_i}-1\ge\frac{1}{a}
\]
so that
\[
am_1\le a^2\le a(a+1)=F_1(a+1)
\]
since $a_k\ge b_k$. This completes the proof of the case $R=1$.

We next assume that the assertion holds for any sequence $\mathcal{M}$ of length $\le R-1$
and prove the assertion for the case in which $\mathcal{M}$ has the length $R$.
We use a special sequence
\[
1<x_{1}<\cdots<x_{R},
\]
which is defined by
\[
x_{j}=
\left\{
\begin{array}{ll}
(a+1)^{2^{j-1}}+1&(\text{for $1\le j< R$})\\
(a+1)^{2^{R-1}}&(\text{for $j=R$}).
\end{array}
\right.
\]
We first consider the case
\[
\prod_{j=1}^{r}x_{j}>\prod_{j=1}^{r}m_{j},
\]
for some $1\le r<R$.
Then we have
\begin{equation}
\label{EQ:acm_small_m1}
am_{1}\cdots m_{r}
<ax_{1}\cdots x_{r}
=
(a+1)^{2^{r}}-1.
\end{equation}
By using notations
\[
a'_i=a_i\prod_{\substack{j=1\\j\in J_i}}^{r}m_{j},\quad
b'_i=b_i\prod_{\substack{j=1\\j\in J_i}}^{r}(m_{j}-1),\quad
a'=a'_1\cdots a'_k,
\]
we can rewrite \eqref{EQ:HB_assump_A1} and \eqref{EQ:HB_assump_B1} as
\[
\sum_{i=1}^{k}\frac{b'_i}{a'_i}\prod_{\substack{j=r+1\\j\in J_i}}^{R}\left(1-\frac{1}{m_j}\right)\le1,\quad
\sum_{i=1}^{k}\frac{b'_i}{a'_i}\prod_{\substack{j=r+1\\j\in J_i}}^{R-1}\left(1-\frac{1}{m_j}\right)>1.
\]
Note that the condition $r<R$ is necessary
for rewriting the inequality \eqref{EQ:HB_assump_B1} as above.
By the induction hypothesis, we obtain
\[
a\prod_{j=1}^{R}m_j
=
a'\prod_{j=r+1}^{R}m_j
\le
F_{R-r}(a'+1)
=
F_{R-r}(am_{1}\cdots m_{r}+1)
\]
By \eqref{EQ:acm_small_m1} and Lemma \ref{Lem:F_increasing}, this implies
\[
a\prod_{j=1}^{R}m_j
\le
F_{R-r}\left((a+1)^{2^{r}}\right)
=
F_{R}(a+1)
\]
so the assertion follows.
Thus we may assume
\begin{equation}
\label{EQ:partial_product_small1}
\prod_{j=1}^{r}x_j\le\prod_{j=1}^{r}m_j
\end{equation}
for all $1\le r<R$.
We may also assume \eqref{EQ:partial_product_small1}
for the case $r=R$ since otherwise
\begin{equation}
\label{EQ:xy_calculation1}
a\prod_{j=1}^{R}m_j\le a\prod_{j=1}^{R}x_j=F_R(a+1)
\end{equation}
and the assertion follows.
Thus, for the remaining case,
we have \eqref{EQ:partial_product_small1} for every $1\le r\le R$.
Then we can apply Lemma \ref{Lem:Cook} to obtain
\begin{equation}
\label{EQ:Cook_M1}
\prod_{j=1}^{R}\left(1-\frac{1}{m_{j}}\right)
\ge\prod_{j=1}^{R}\left(1-\frac{1}{x_{j}}\right)
=\frac{a}{a+1}
\end{equation}
Then by \eqref{EQ:HB_assump_A1} and \eqref{EQ:abcd_ineq1}, we have
\[
1
\ge
\sum_{i=1}^{k}\frac{b_i}{a_i}\prod_{\substack{j=1\\j\in J_i}}^{R}\left(1-\frac{1}{m_{j}}\right)
\ge
\prod_{j=1}^{R}\left(1-\frac{1}{m_{j}}\right)\sum_{i=1}^{k}\frac{b_i}{a_i}
\ge
\frac{a}{a+1}\sum_{i=1}^{k}\frac{b_i}{a_i}\ge1.
\]
Thus we must have the equality in \eqref{EQ:Cook_M1}.
By Lemma \ref{Lem:Cook}, we find that
\[
m_{1}=x_{1},\ \ \ldots,\ \ m_{R}=x_{R}.
\]
By using \eqref{EQ:xy_calculation1}, we have the assertion again.
This completes the proof.
\end{proof}

%
%
We next prove the second lemma on Diophantine inequality.
\begin{lemma}
\label{Lem:HB_ineq2}
Let $k\ge1$ be an integer.
For $R\ge1$, consider a sequence of integers
\[
\mathcal{M}=(m_j)_{j=1}^{R},\quad1<m_1\le m_2\le\cdots\le m_R,
\]
a decomposition of the index set
\[
J=\{1,\,\ldots,\,R\},\quad
J=\bigcup_{i=1}^{k}J_i,\quad
J_1,\,\ldots,\,J_k\colon\text{disjoint}
\]
and a tuple of integers
\[a_1,\,\ldots,\,a_k,\,b_1,\,\ldots,\,b_k\ge1.\]
If a pair of inequalities
\begin{equation}
\label{EQ:HB_assump_A2}
\sum_{i=1}^{k}\frac{b_i}{a_i}\prod_{\substack{j=1\\j\in J_i}}^{R}\left(1-\frac{1}{m_j}\right)^{-1}\ge1
\end{equation}
\begin{equation}
\label{EQ:HB_assump_B2}
\sum_{i=1}^{k}\frac{b_i}{a_i}\prod_{\substack{j=1\\j\in J_i}}^{R-1}\left(1-\frac{1}{m_j}\right)^{-1}<1
\end{equation}
holds, then we have
\begin{equation}
\label{EQ:M_bound1}
a\prod_{j=1}^{R}(m_j-1)\le F_{R}(a),
\end{equation}
where $a=a_1\cdots a_k$.
\end{lemma}
\begin{remark}
\label{Rem:abcd_cond}
We assumed $a_i\ge b_i$ in Lemma \ref{Lem:HB_ineq1},
but we do not assume this condition in Lemma \ref{Lem:HB_ineq2} above.
Actually, by \eqref{EQ:HB_assump_B2}, we automatically obtain $a_i>b_i$.
\end{remark}
\begin{proof}
By \eqref{EQ:HB_assump_B2}, we always have
\begin{equation}
\label{EQ:abcd_ineq2}
\frac{1}{a}\le\sum_{i=1}^{k}\frac{b_i}{a_i}\le1-\frac{1}{a}
\end{equation}
as in the proof of \eqref{EQ:abcd_ineq1}.
Again this is a key in the argument below.

We use induction on $R$.
If $R=1$, by symmetry we may assume without loss of generality that
$J_1=\cdots=J_{k-1}=\emptyset$ and $J_{k}=\{1\}$.
Then \eqref{EQ:HB_assump_A2} implies
\[
1
\le
\sum_{i=1}^{k-1}\frac{b_i}{a_i}+\frac{b_k}{a_k}\left(1-\frac{1}{m_1}\right)^{-1}
=
\sum_{i=1}^{k}\frac{b_i}{a_i}+\frac{b_k}{a_k}\cdot\frac{1}{m_1-1}.
\]
By \eqref{EQ:abcd_ineq2}, we find that
\[
\frac{b_k}{a_k}\cdot\frac{1}{m_1-1}\ge1-\sum_{i=1}^{k}\frac{b_i}{a_i}\ge\frac{1}{a}
\]
so that
\[
a(m_1-1)\le a(a-1)=F_1(a)
\]
since $a_k>b_k$.
This completes the proof of the case $R=1$.

We next assume that the assertion holds for any sequence $\mathcal{M}$ of length $\le R-1$
and prove the assertion for the case in which $\mathcal{M}$ has the length $R$.
We use a special sequence
\[
1<x_{1}<\cdots<x_{R},
\]
which is defined by
\[
x_{j}=
\left\{
\begin{array}{ll}
a^{2^{j-1}}+1&(\text{for $1\le j< R$})\\
a^{2^{R-1}}&(\text{for $j=R$}).
\end{array}
\right.
\]
We first consider the case
\[
\prod_{j=1}^{r}(x_{j}-1)>\prod_{j=1}^{r}(m_{j}-1),
\]
for some $1\le r<R$. Then we have
\begin{equation}
\label{EQ:acm_small_m2}
a(m_{1}-1)\cdots (m_{r}-1)
<a(x_{1}-1)\cdots (x_{r}-1)
=
a^{2^{r}}.
\end{equation}
By using notations
\[
a''_i=a_i\prod_{\substack{j=1\\j\in J_i}}^{r}(m_{j}-1),\quad
b''_i=b_i\prod_{\substack{j=1\\j\in J_i}}^{r}m_{j},\quad
a''=a''_1\cdots a''_k,
\]
we can rewrite \eqref{EQ:HB_assump_A2} and \eqref{EQ:HB_assump_B2} as
\[
\sum_{i=1}^{k}\frac{b''_i}{a''_i}
\prod_{\substack{j=r+1\\j\in J_i}}^{R}\left(1-\frac{1}{m_j}\right)^{-1}\ge1,\quad
\sum_{i=1}^{k}\frac{b''_i}{a''_i}
\prod_{\substack{j=r+1\\j\in J_i}}^{R-1}\left(1-\frac{1}{m_j}\right)^{-1}<1.
\]
By the induction hypothesis and \eqref{EQ:acm_small_m2}, we obtain
\begin{align*}
a\prod_{j=1}^{R}(m_j-1)
&=
a''\prod_{j=r+1}^{R}(m_j-1)
\le
F_{R-r}(a'')
\le
F_{R}(a)
\end{align*}
so the assertion follows.
Thus we may assume
\begin{equation}
\label{EQ:partial_product_small2}
\prod_{j=1}^{r}(x_j-1)\le\prod_{j=1}^{r}(m_j-1)
\end{equation}
for all $1\le r<R$. We may also assume \eqref{EQ:partial_product_small2}
for the case $r=R$ since otherwise
\begin{equation}
\label{EQ:xy_calculation2}
a\prod_{j=1}^{R}(m_j-1)\le a\prod_{j=1}^{R}(x_j-1)=F_R(a)
\end{equation}
and the assertion follows.
Thus, for the remaining case,
we have \eqref{EQ:partial_product_small2} for every $1\le r\le R$.
Note that by \eqref{EQ:abcd_ineq2},
\[
1<a=(x_1-1).
\]
Thus we can apply Lemma \ref{Lem:Cook} to obtain
\begin{equation}
\label{EQ:Cook_M2}
\prod_{j=1}^{R}\left(1-\frac{1}{m_{j}}\right)^{-1}
=
\prod_{j=1}^{R}\left(1+\frac{1}{m_{j}-1}\right)
\le
\prod_{j=1}^{R}\left(1+\frac{1}{x_{j}-1}\right)
=\frac{a}{a-1}.
\end{equation}
Then by \eqref{EQ:HB_assump_A2} and \eqref{EQ:abcd_ineq2}, we have
\[
1
\le
\sum_{i=1}^{k}\frac{b_i}{a_i}\prod_{\substack{j=1\\j\in J_i}}^{R}\left(1-\frac{1}{m_{j}}\right)^{-1}
\le
\prod_{j=1}^{R}\left(1-\frac{1}{m_{j}}\right)^{-1}\sum_{i=1}^{k}\frac{b_i}{a_i}
\le
\frac{a}{a-1}\sum_{i=1}^{k}\frac{b_i}{a_i}
\le1.
\]
Thus we must have the equality in \eqref{EQ:Cook_M2}.
By Lemma \ref{Lem:Cook}, we find that
\[
m_{1}=x_{1},\ \ \ldots,\ \ m_{R}=x_{R}.
\]
By using \eqref{EQ:xy_calculation2}, we have the assertion again.
This completes the proof.
\end{proof}

%
%
\section{Upper bounds \`a la Borho}
\label{Section:Omega}
In this section,
we prove Theorem 4, which gives upper bounds of Borho-type.
\begin{proof}[Proof of Theorem \ref{Thm:Omega}]
We first consider a harmonious tuple $(M_{i})_{i=1}^{k}$.
Note that
\[
\frac{M}{\sigma(M)}
=
\prod_{p^e\parallel M}\frac{p^e}{1+\cdots+p^e}\\
=
\prod_{p^e\parallel M}\prod_{f=1}^{e}\left(1-\frac{1}{1+\cdots+p^f}\right)
\]
as it is mentioned in the proof of Satz 3 of \cite{Borho2}. Thus, \eqref{EQ:fundamental_eq}
is rewritten as
\[
\sum_{i=1}^{k}\prod_{p^e\parallel M_i}\prod_{f=1}^{e}\left(1-\frac{1}{1+\cdots+p^f}\right)=1.
\]
If we remove any factor from any summand, then the left-hand side becomes larger.
Thus we can apply Lemma \ref{Lem:HB_ineq1} and obtain
\begin{equation}
\label{EQ:Omega_after_lemma}
\sigma(M_1)\cdots\sigma(M_k)
\le
\prod_{i=1}^{k}\prod_{p^e\parallel M_i}\prod_{f=1}^{e}(1+\cdots+p^f)
\le
F_L(2)=2^{2^{L}}-2^{2^{L-1}},
\end{equation}
where $L$ is given by the number of factors in the product above, so
\[
L=\sum_{i=1}^{k}\sum_{p^e\parallel M}e=\sum_{i=1}^{k}\Omega(M_i)=\Omega(M_1\cdots M_k).
\]
Note that $M_i$ can share some common factor since we do not assume anything on the divisibility.
However, this does not affect the above arguments.
Now by using the inequality of the arithmetic and geometric
mean in \eqref{EQ:fundamental_eq}, we find that
\[
1
=\frac{M_1}{\sigma(M_1)}+\cdots+\frac{M_k}{\sigma(M_k)}
\ge k\left(\frac{M_1\cdots M_k}{\sigma(M_1)\cdots\sigma(M_k)}\right)^{1/k}
\]
so
\[
\sigma(M_1)\cdots\sigma(M_k)
=
\left(\frac{\sigma(M_1)\cdots\sigma(M_k)}{M_1\cdots M_k}\right)M_1\cdots M_k
\ge
k^kM_1\cdots M_k.
\]
On inserting this into \eqref{EQ:Omega_after_lemma},
we arrive at the assertion for amicable tuples.

We next consider a unitary harmonious tuple $(M_{i})_{i=1}^{k}$.
By using
\[
\frac{M}{\sigma^\ast(M)}=\prod_{p^e\parallel M}\frac{p^e}{1+p^e}=\prod_{p^e\parallel M}\left(1-\frac{1}{1+p^e}\right)
\]
we can rewrite \eqref{EQ:fundamental_eq_unitary} as
\[
\sum_{i=1}^{k}\prod_{p^e\parallel M_i}\left(1-\frac{1}{1+p^e}\right)=1.
\]
Applying Lemma \ref{Lem:HB_ineq1} as above, we see that
\begin{equation}
\label{EQ:unitary_after_lemma}
\sigma^\ast(M_1)\cdots\sigma^\ast(M_k)
=
\prod_{i=1}^{k}\prod_{p^e\parallel M_i}(1+p^e)
\le
F_L(2)=2^{2^{L}}-2^{2^{L-1}},
\end{equation}
where $L$ is given by
\[
L=\sum_{i=1}^{k}\sum_{p^e\parallel M_i}1=\omega(M_1)+\cdots+\omega(M_k).
\]
By using the inequality of the arithmetic and geometric
mean in \eqref{EQ:fundamental_eq_unitary}, we find
\[
\sigma^\ast(M_1)\cdots\sigma^\ast(M_k)
=
\left(\frac{\sigma^\ast(M_1)\cdots\sigma^\ast(M_k)}{M_1\cdots M_k}\right)M_1\cdots M_k
\ge
k^kM_1\cdots M_k
\]
On inserting this into \eqref{EQ:unitary_after_lemma},
we obtain the assertion for unitary harmonious pairs.
\end{proof}

%
%
\section{The induction lemma}
\label{Section:induction}
In this section, we prove an induction lemma.
We start with a lemma on the divisibility,
whose special case is also used in the proof of Lemma 4 of \cite{Pollack_bound}.
%
%
\begin{lemma}
\label{Lem:divisibility_lemma}
Let $k\ge2$ be an integer, $(M_i)_{i=1}^{k}$ be an anarchy harmonious tuple
and suppose that a tuple of decompositions
\[
M_i=U_iV_i,\quad
(U_i,V_i)=1,\quad
U:=U_{1}\cdots U_{k}>1
\]
is given.
Then we have
\[
\sum_{i=1}^{k}
\frac{V_i}{\sigma(V_i)}
\prod_{\substack{p\mid U_i\\p\in\mathcal{S}}}\left(1-\frac{1}{p}\right)
\neq1
\]
for any set $\mathcal{S}$ of prime factors of $U$.
\end{lemma}
\begin{proof}
Assume to the contrary that
\begin{equation}
\label{EQ:assumption_divisibility_lemma}
\sum_{i=1}^{k}
\frac{V_i}{\sigma(V_i)}
\prod_{\substack{p\mid U_i\\p\in\mathcal{S}}}\left(1-\frac{1}{p}\right)
=1
\end{equation}
for some set $\mathcal{S}$ of prime factors of $U$.
We first claim that $\mathcal{S}$ is non-empty.
Since $U=U_{1}\cdots U_{k}>1$, we have $U_{i}/\sigma(U_{i})<1$ for some $i$.
Thus, by \eqref{EQ:fundamental_eq},
\[
1=\sum_{i=1}^{k}\frac{U_i}{\sigma(U_i)}\frac{V_i}{\sigma(V_i)}<\sum_{i=1}^{k}\frac{V_i}{\sigma(V_i)}.
\]
Comparing this with \eqref{EQ:assumption_divisibility_lemma}, $\mathcal{S}$ should be non-empty.
By multiplying \eqref{EQ:assumption_divisibility_lemma} by
\[
\prod_{i=1}^{k}\sigma(V_i)\prod_{p\in\mathcal{S}}p,
\]
we have
\begin{equation}
\label{EQ:after_multiplying}
\sum_{i=1}^{k}
V_i\prod_{\substack{j=1\\i\neq j}}^{k}\sigma(V_j)
\prod_{\substack{p\mid U_i\\p\in\mathcal{S}}}(p-1)
\prod_{\substack{p\nmid U_i\\p\in\mathcal{S}}}p
=
\prod_{i=1}^{k}\sigma(V_i)\prod_{p\in\mathcal{S}}p.
\end{equation}
Let $P$ be the largest prime in $\mathcal{S}$, which exists since $\mathcal{S}$ is non-empty.
By symmetry, we may assume $P\mid U_1$.
Then in \eqref{EQ:after_multiplying},
all terms except the case $i=1$ on the left-hand side and the right-hand side
are divisible by $P$ since $(M_i)_{i=1}^{k}$ is anarchy so $U_{1}.\ldots,U_{k}$ are pairwise coprime.
This implies
\begin{equation}
\label{EQ:P_divides}
P\mid V_1\prod_{j=2}^{k}\sigma(V_j)
\prod_{\substack{p\mid U_1\\p\in\mathcal{S}}}(p-1)
\prod_{\substack{p\nmid U_1\\p\in\mathcal{S}}}p
\end{equation}
Since $P$ is the largest prime in $\mathcal{S}$ and $(M_i)_{i=1}^{k}$ is anarchy, we find
\[
P\nmid \prod_{j=2}^{k}\sigma(V_j)
\prod_{\substack{p\mid U_1\\p\in\mathcal{S}}}(p-1)\prod_{\substack{p\nmid U_1\\p\in\mathcal{S}}}p.
\]
Thus by \eqref{EQ:P_divides}, $P\mid V_1$,
which contradicts to $P\mid U_1$ and $(U_1,V_1)=1$.
\end{proof}

%
%
By using Lemma \ref{Lem:divisibility_lemma},
we can prove now the following variant of Heath-Brown's induction lemma,
which corresponds to Lemma 1.5 of \cite{Nielsen_New}.
\begin{lemma}
\label{Lem:HB_induction}
Let $k\ge2$ be an integer, $(M_i)_{i=1}^{k}$ be an anarchy harmonious tuple
and suppose that a tuple of decompositions
\[
M_i=U_iV_i,\quad
(U_i,V_i)=1,\quad
U:=U_{1}\cdots U_{k}>1,\quad
V:=V_{1}\cdots V_{k},
\]
and a set $\mathcal{S}$ of prime factors of $U$ are given.
Then there is a tuple of decompositions
\[
M'_i=U'_iV'_i,\quad
(U'_i,V'_i)=1,\quad
U':=U'_{1}\cdots U'_{k},\quad
V':=V'_{1}\cdots V'_{k},\quad
V\parallel V',
\]
and a set $\mathcal{S}'$ of prime factors of $U'$ with the following conditions\,{\upshape:}
\begin{enumerate}
\renewcommand{\labelenumi}{{\upshape(\roman{enumi})}}
\item $v:=|\mathcal{P}'|\ge1$,
where $\mathcal{P}':=\{p\colon\text{prime}\mid p\mid V',p\nmid V\}$,
\item we have
\[
\sigma(V')\Pi(\mathcal{S}')\Psi(\mathcal{P}')
\le
F_{v+w}(\sigma(V)\Pi(\mathcal{S})+1),
\]
where $w:=v+|\mathcal{S}'|-|\mathcal{S}|$,
\item if $w=0$, then the inequality in {\upshape(ii)} can be improved to
\[
\sigma(V')\Pi(\mathcal{S}')\Psi(\mathcal{P}')
\le
F_{v+w}(\sigma(V)\Pi(\mathcal{S})).\hphantom{{}+{}1}
\]
\end{enumerate}
\end{lemma}
\begin{proof}
We first show that there is a set $\mathcal{T}$ of prime factors of $U$ satisfying
\begin{equation}
\label{EQ:T_disjoint_cond}
\mathcal{S}\cap\mathcal{T}=\emptyset,
\end{equation}
\begin{equation}
\label{EQ:T_size_cond}
\sum_{i=1}^{k}
\frac{V_i}{\sigma(V_i)}
\prod_{\substack{p\mid U_i\\p\in\mathcal{S}\cup\mathcal{T}}}\left(1-\frac{1}{p}\right)<1
\end{equation}
and
\begin{equation}
\label{EQ:T_bound}
\sigma(V)\Pi(\mathcal{S})\Pi(\mathcal{T})\le F_{w}(\sigma(V)\Pi(\mathcal{S})+1),\quad
w:=|\mathcal{T}|.
\end{equation}
This $w$ will be the same quantity as in the condition (ii).
By Lemma \ref{Lem:divisibility_lemma}, the quantity
\begin{equation}
\label{EQ:S_form}
H:=
\sum_{i=1}^{k}\frac{V_i}{\sigma(V_i)}\prod_{\substack{p\mid U_i\\p\in\mathcal{S}}}\left(1-\frac{1}{p}\right)
\end{equation}
never equals $1$. We consider two cases separately according to the size of $H$.

If $H<1$, we just take $\mathcal{T}=\emptyset$ so that $w=0$.
This choice obviously satisfies the conditions
\eqref{EQ:T_disjoint_cond}, \eqref{EQ:T_size_cond} and \eqref{EQ:T_bound}.
Thus the case $H<1$ is done.

We next consider the case $H>1$.
Since $U>1$, we have
\[
\frac{U_i}{\sigma(U_i)}>\prod_{p\mid U_i}\left(1-\frac{1}{p}\right)
\]
for some $i$.
Thus, by \eqref{EQ:fundamental_eq}, we see that
\begin{equation}
\label{EQ:beyond}
\sum_{i=1}^{k}\frac{V_i}{\sigma(V_i)}\prod_{p\mid U_i}\left(1-\frac{1}{p}\right)
<1.
\end{equation}
By using notation
\[
a_i=\sigma(V_i)\prod_{\substack{p\mid U_i\\p\in\mathcal{S}}}p,\quad
b_i=V_i\prod_{\substack{p\mid U_i\\p\in\mathcal{S}}}(p-1),
\]
we can rewrite \eqref{EQ:beyond} as
\[
\sum_{i=1}^{k}
\frac{b_i}{a_i}\prod_{\substack{p\mid U_i\\p\not\in\mathcal{S}}}\left(1-\frac{1}{p}\right)
<1.
\]
Note that $a_i\ge b_i$ for all $i$.
Thus, by comparing this inequality with $H>1$ and examining from the smallest prime factors
of $U$ outside $\mathcal{S}$, we can find a non-empty set $\mathcal{T}=\{p_1,\ldots,p_w\}$
of prime factors of $U$ with $p_1<\cdots<p_w$,
which satisfies $\mathcal{S}\cap\mathcal{T}=\emptyset$ and two inequalities
\begin{equation}
\label{EQ:HB_assump_T}
\sum_{i=1}^{k}
\frac{b_i}{a_i}\prod_{\substack{j=1\\p\mid U_i}}^{w}\left(1-\frac{1}{p_j}\right)
\le1,\quad
\sum_{i=1}^{k}
\frac{b_i}{a_i}\prod_{\substack{j=1\\p\mid U_i}}^{w-1}\left(1-\frac{1}{p_j}\right)
>1.
\end{equation}
By applying Lemma \ref{Lem:HB_ineq1} to this pair of inequalities,
we find that
\[
a\Pi(\mathcal{T})=a\prod_{j=1}^{w}p_j\le F_w(a+1),\quad
a=a_1\cdots a_k=\sigma(V)\Pi(\mathcal{S}),
\]
i.e.~\eqref{EQ:T_bound} holds. For \eqref{EQ:T_size_cond},
we expand the definition of $a_i$ and $b_i$ in \eqref{EQ:HB_assump_T} to obtain
\[
\sum_{i=1}^{k}
\frac{V_i}{\sigma(V_i)}\prod_{\substack{p\mid U_i\\p\in\mathcal{S}\cup\mathcal{T}}}\left(1-\frac{1}{p}\right)
\le1.
\]
Then this equality cannot hold by Lemma \ref{Lem:divisibility_lemma}
so the condition \eqref{EQ:T_size_cond} holds.
Therefore, in any case, we succeeded to find a set $\mathcal{T}$ satisfying
the desired conditions.

We next show that there is a non-empty subset $\mathcal{P}'$ of $\mathcal{S}\cup\mathcal{T}$
which satisfies
\begin{equation}
\label{EQ:P_bound}
\sigma(V)\Pi(\mathcal{S})\Pi(\mathcal{T})
\prod_{\substack{p^e\parallel U\\p\in\mathcal{P}'}}(p^{e+1}-1)
\le
F_{v}(\sigma(V)\Pi(\mathcal{S})\Pi(\mathcal{T})),\quad
v=|\mathcal{P}'|.
\end{equation}
These $\mathcal{P}'$ and $v$ will be the same objects as in the condition (i).
Since $n/\sigma(n)\le 1$ for any positive integer $n$,
\eqref{EQ:fundamental_eq} implies
\begin{equation}
\label{EQ:beyond2}
\sum_{i=1}^{k}
\frac{V_i}{\sigma(V_i)}
\prod_{\substack{p^e\parallel U_i\\p\in\mathcal{S}\cup\mathcal{T}}}\frac{1-1/p}{1-1/p^{e+1}}
\ge1
\end{equation}
By using notations
\[
a'_i=\sigma(V_i)\!\!\!\prod_{\substack{p\mid U_i\\p\in\mathcal{S}\cup\mathcal{T}}}\!\!\!p,\quad
b'_i=V_i\!\!\!\prod_{\substack{p\mid U_i\\p\in\mathcal{S}\cup\mathcal{T}}}\!\!\!(p-1),
\]
we can rewrite \eqref{EQ:beyond2} as
\begin{equation}
\label{EQ:beyond3}
\sum_{i=1}^{k}
\frac{b'_i}{a'_i}
\prod_{\substack{p^e\parallel U_i\\p\in\mathcal{S}\cup\mathcal{T}}}
\left(1-\frac{1}{p^{e+1}}\right)^{-1}
\ge1.
\end{equation}
Similarly, \eqref{EQ:T_size_cond} can be rewritten as
\begin{equation}
\label{EQ:T_size_cond2}
\sum_{i=1}^{k}
\frac{b'_i}{a'_i}<1.
\end{equation}
By comparing \eqref{EQ:beyond3} and \eqref{EQ:T_size_cond2}
and examining from the smallest values of
\[
p^{e+1},\quad p\in\mathcal{S}\cup\mathcal{T},\ p^e\parallel U,
\]
we can find a non-empty subset $\mathcal{P}'=\{P_1,\ldots,P_v\}$ of $\mathcal{S}\cup\mathcal{T}$
with
\[
P_1^{e_1+1}<\cdots<P_v^{e_v+1},\quad
P_j^{e_j}\parallel U
\]
satisfying two inequalities
\begin{equation}
\label{EQ:HB_assump_P}
\sum_{i=1}^{k}
\frac{b'_i}{a'_i}\prod_{\substack{j=1\\P_j|U_i}}^{v}\left(1-\frac{1}{P_j^{e_j+1}}\right)^{-1}
\ge1,\quad
\sum_{i=1}^{k}
\frac{b'_i}{a'_i}\prod_{\substack{j=1\\P_j|U_i}}^{v-1}\left(1-\frac{1}{P_j^{e_j+1}}\right)^{-1}
<1.
\end{equation}
Then by applying Lemma \ref{Lem:HB_ineq2} to \eqref{EQ:HB_assump_P}, we find that
\[
a'
\prod_{j=1}^{v}(P_j^{e_j+1}-1)
\le
F_{v}(a'),\quad
a'=a'_1\cdots a'_k=\sigma(V)\Pi(\mathcal{S})\Pi(\mathcal{T}),
\]
i.e.~\eqref{EQ:P_bound} holds. Thus our $\mathcal{P}'$ satisfies the desired condition.

Finally, we choose
\[
\mathcal{S}'=(\mathcal{S}\cup\mathcal{T})\setminus\mathcal{P}',\quad
V'_i=V_i\prod_{\substack{p^e\parallel U_i\\p\in\mathcal{P}'}}p^e,\quad
M_{i}=U'_{i}V'_{i}.
\]
Then it is clear that $(U'_{i},V'_{i})=1$, $V\parallel V'$.
The set $\mathcal{S}'$ consists of some prime factors of $U'$ since
the prime factors of $U'$ are those of $U$ outside the set $\mathcal{P}'$.

The remaining task is to check the conditions (i), (ii) and (iii).
Note that the notations on $\mathcal{P}'$ and $v$ keep its consistency.
Since $\mathcal{P}'$ is non-empty, the condition (i) of the lemma is satisfied.
For the consistency on $w$, it suffices to see
\[
v+|\mathcal{S}'|-|\mathcal{S}|
=
|\mathcal{P}'|+|\mathcal{S}\cup\mathcal{T}|-|\mathcal{P}'|-|\mathcal{S}|
=
|\mathcal{S}\cup\mathcal{T}|-|\mathcal{S}|
=
|\mathcal{T}|.
\]
We prove the inequality in (ii) and (iii).
By our choice of $V'_i$ and $\mathcal{S}'$,
\begin{align*}
\sigma(V')\Pi(\mathcal{S}')\Pi(\mathcal{P}')\Phi(\mathcal{P}')
&=
\sigma(V)\sigma(\prod_{\substack{p^e\parallel U\\p\in\mathcal{P}'}}p^e)
\Pi(\mathcal{S})\Pi(\mathcal{T})\Phi(\mathcal{P}')\\
&=
\sigma(V)\Pi(\mathcal{S})\Pi(\mathcal{T})
\prod_{\substack{p^e\parallel U\\p\in\mathcal{P}'}}(p^{e+1}-1)
\end{align*}
By \eqref{EQ:P_bound} and the definition of $\Psi$, this implies
\[
\sigma(V')\Pi(\mathcal{S}')\Psi(\mathcal{P}')
\le
F_{v}(\sigma(V)\Pi(\mathcal{S})\Pi(\mathcal{T})).
\]
If $w=0$, then $\mathcal{T}=\emptyset$ so that this inequality already gives the inequality in (iii).
We substitute \eqref{EQ:T_bound} here.
Then we arrive at
\begin{align*}
\sigma(V')\Pi(\mathcal{S}')\Psi(\mathcal{P}')
&\le
F_{v}(F_{w}(\sigma(V)\Pi(\mathcal{S})+1))\\
&\le
F_{v}((\sigma(V)\Pi(\mathcal{S})+1)^{2^{w}})\\
&=
F_{v+w}(\sigma(V)\Pi(\mathcal{S})+1).
\end{align*}
Thus the inequality in (ii) also holds.
This completes the proof.
\end{proof}

%
%
\section{Completion of the proof of Theorem \ref{Thm:main}}
We start with carrying out the induction given in Section \ref{Section:induction}.

%
%
\begin{lemma}
\label{Lem:carry_out_induction}
For any anarchy harmonious tuple $(M_{i})_{i=1}^{k}$ with $\omega(M_1\cdots M_k)=K$,
\[
\sigma(M_1\cdots M_k)\frac{\Phi(\mathcal{P})}{\Pi(\mathcal{P})}\le F_{2K}(2)\Pi(\mathcal{P})^{-2},
\]
where $\mathcal{P}$ is the set of all prime factors of $M_1\cdots M_k$.
\end{lemma}
\begin{proof}
Let $(M_{i})_{i=1}^{k}$ be an anarchy harmonious tuple with $K=\omega(M_1\cdots M_k)$.
We apply Lemma \ref{Lem:HB_induction} inductively to construct tuples of decompositions
\[
M_{i}(\nu)=U_{i}(\nu)V_{i}(\nu),\quad(U_{i}(\nu),V_{i}(\nu))=1,\quad(\nu=0,1,2,\ldots)
\]
and sets of primes
\[
\mathcal{S}(\nu),\ \mathcal{P}(\nu),\quad(\nu=0,1,2,\ldots),
\]
where $\mathcal{S}(\nu)$ is a set of prime factors of $U(\nu):=U_1(\nu)\cdots U_k(\nu)$.
We start with
\[
U_{i}(0)=M_i,\ V_{i}(0)=1,\quad\mathcal{S}(0)=\emptyset,\ \mathcal{P}(0)=\emptyset.
\]
In general steps, we apply Lemma \ref{Lem:HB_induction} to the $\nu$-th term with
\[
U_{i}=U_{i}(\nu),\ V_{i}=V_{i}(\nu),\quad\mathcal{S}=\mathcal{S}(\nu)
\]
and define the $(\nu+1)$-th term by
\[
U_{i}(\nu+1)=U'_{i},\ V_{i}(\nu+1)=V'_i,\quad
\mathcal{S}(\nu+1)=\mathcal{S}',\ \mathcal{P}(\nu+1)=\mathcal{P}'.
\]
Then we can continue this induction step as long as $U(\nu)>1$.
Let
\[
v(\nu)=|\mathcal{P}(\nu)|,\quad
w(\nu)=v(\nu)+|\mathcal{S}(\nu)|-|\mathcal{S}(\nu-1)|
\]
for $\nu\ge1$ as long as the induction step is available.
By the definition of $\mathcal{P}(\nu)$,
\begin{equation}
\label{EQ:P_partition}
\mathcal{P}(\nu)=\{p\colon\text{prime}\mid p\mid V(\nu)\}\setminus
\{p\colon\text{prime}\mid p\mid V(\nu-1)\},
\end{equation}
where $V(\nu):=V_1(\nu)\cdots V_k(\nu)$, so
\[
v(\nu)=\omega(V(\nu))-\omega(V(\nu-1))
\]
since $V(\nu-1)\parallel V(\nu)$.
Thus, by (i) of Lemma~\ref{Lem:HB_induction},
\begin{equation}
\label{EQ:v_sum}
s\le v(1)+\cdots+v(s)=\omega(V(s))\le \omega(M_1\cdots M_k)=K
\end{equation}
for any $s\ge1$ if the induction step is available until the $(s-1)$-th step.
Thus the induction step stops in finitely many steps.
Suppose that the induction step stops at the $n$-th step to produce
\[
U_{i}(n)=1,\ V_{i}(n)=M_i,\quad\mathcal{S}(n)=\emptyset.
\]
Thus by (iii) of Lemma~\ref{Lem:HB_induction}, for $2\le s\le n$ satisfying $w(s)=0$, we have 
\begin{align*}
\sigma(V(s))\Pi(\mathcal{S}(s))\Psi(\mathcal{P}(s))
\le
F_{v(s)+w(s)}\left(\sigma(V(s-1))\Pi(\mathcal{S}(s-1))\right)
\end{align*}
Since $v(s)+w(s)=v(s)\ge1$, Lemma \ref{Lem:F_scaling} gives
\begin{equation}
\label{EQ:induction_bound}
\begin{gathered}
\sigma(V(s))\Pi(\mathcal{S}(s))\Psi(\mathcal{P}(s))\\
\le
\Psi(\mathcal{P}(s-1))^{-1}
F_{v(s)+w(s)}\left(\sigma(V(s-1))\Pi(\mathcal{S}(s-1))\Psi(\mathcal{P}(s-1))\right).
\end{gathered}
\end{equation}
For remaining $2\le s\le n$ with $w(s)\ge1$, we use (ii) of Lemma~\ref{Lem:HB_induction} to obtain
\begin{align*}
\sigma(V(s))\Pi(\mathcal{S}(s))\Pi(\mathcal{P}(s))\Phi(\mathcal{P}(s))
&\le
F_{v(s)+w(s)}\left(\sigma(V(s-1))\Pi(\mathcal{S}(s-1))+1\right)\\
&\le
F_{v(s)+w(s)}\left(\frac{4}{3}\sigma(V(s-1))\Pi(\mathcal{S}(s-1))\right)
\end{align*}
since $\sigma(V(s-1))\ge 3$ for $s\ge2$.
Also, since $\mathcal{P}(s-1)\neq\emptyset$ for $s\ge2$, we find that
\[
\Psi(\mathcal{P}(s-1))=\Pi(\mathcal{P}(s-1))\Phi(\mathcal{P}(s-1))\ge2\ge\left(\frac{4}{3}\right)^2.
\]
Therefore, by Lemma \ref{Lem:F_scaling},
\begin{gather*}
\sigma(V(s))\Pi(\mathcal{S}(s))\Psi(\mathcal{P}(s))\\
\le
\Psi(\mathcal{P}(s-1))^{-2^{v(s)+w(s)-2}}
F_{v(s)+w(s)}\left(\sigma(V(s-1))\Pi(\mathcal{S}(s-1))\Psi(\mathcal{P}(s-1))\right)
\end{gather*}
so by using $v(s)+w(s)\ge v(s)+1\ge2$, we again arrive at \eqref{EQ:induction_bound}.
Thus the estimate \eqref{EQ:induction_bound} holds for every $2\le s\le n$.
Then by using  \eqref{EQ:induction_bound} inductively,
\begin{align*}
&\sigma(V(n))\Pi(\mathcal{S}(n))\Psi(\mathcal{P}(n))\\
{}\le{}&%
\Psi(\mathcal{P}(n-1))^{-1}
F_{v(n)+w(n)}\left(\sigma(V(n-1))\Pi(\mathcal{S}(n-1))\Psi(\mathcal{P}(n-1))\right)\\
{}\le{}&%
\Psi(\mathcal{P}(n-1))^{-1}\Psi(\mathcal{P}(n-2))^{-2^{v(n)+w(n)}}\\
&\quad\times
F_{v(n)+w(n)+v(n-1)+w(n-1)}\left(\sigma(V(n-2))\Pi(\mathcal{S}(n-2))\Psi(\mathcal{P}(n-2))\right)\\
{}\le{}&%
\Psi(\mathcal{P}(n-1))^{-1}\Psi(\mathcal{P}(n-2))^{-1}\\
&\quad\times
F_{v(n)+w(n)+v(n-1)+w(n-1)}
\left(\sigma(V(n-2))\Pi(\mathcal{S}(n-2))\Psi(\mathcal{P}(n-2))\right)\\
{}\le{}&\cdots\\
{}\le{}&%
\Psi\left(\bigsqcup_{\nu=1}^{n-1}\mathcal{P}(\nu)\right)^{-1}
F_{v(n)+w(n)+\cdots+v(2)+w(2)}\left(\sigma(V(1))\Pi(\mathcal{S}(1))\Psi(\mathcal{P}(1))\right).
\end{align*}
By using (ii) of Lemma \ref{Lem:HB_induction} once more and recalling \eqref{EQ:P_partition},
\begin{align*}
\sigma(M_1\cdots M_k)
&=
\Psi(\mathcal{P}(n))^{-1}\sigma(V(n))\Pi(\mathcal{S}(n))\Psi(\mathcal{P}(n))\\
&\le
\Psi\left(\mathcal{P}\right)^{-1}
F_{\sum_{\nu=1}^{n}(v(\nu)+w(\nu))}(\sigma(V(0))\Pi(\mathcal{S}(0))+1)\\
&=
\Psi\left(\mathcal{P}\right)^{-1}F_{\sum_{\nu=1}^{n}(v(\nu)+w(\nu))}(2).
\end{align*}
By definition of $w(\nu)$ and $\mathcal{S}(0)=\mathcal{S}(n)=\emptyset$,
we have
\begin{align*}
\sum_{\nu=1}^{n}(v(\nu)+w(\nu))
&=2\sum_{\nu=1}^{n}v(\nu)+|\mathcal{S}(n)|-|\mathcal{S}(0)|\\
&=2\sum_{\nu=1}^{n}v(\nu)=2\omega(V(n))=2\omega(M_1\cdots M_k)=2K.
\end{align*}
Thus the lemma follows.
\end{proof}

%
%
We next prove the following auxiliary lemma,
which is an amicable number analogue of the lemma used in the clever trick of Chen and Tang~\cite[Lemma~2.3]{Chen_Tang} or of Kobayashi
(see the remark before Corollary 1.7 of \cite{Nielsen_New}).
\begin{lemma}
\label{Lem:Chen_Tang}
For any anarchy harmonious tuple $(M_{i})_{i=1}^{k}$ with $\omega(M_1\cdots M_k)=K$,
\[
\sigma(M_1\cdots M_k)\frac{\Phi(\mathcal{P})}{\Pi(\mathcal{P})}\le F_K(\Pi(\mathcal{P}))\Pi(\mathcal{P})^{-2},
\]
where $\mathcal{P}$ is the set of all prime factors of $M_1\cdots M_k$.
\end{lemma}
\begin{proof}
Since $(M_{i})_{i=1}^{k}$ is a harmonious tuple, the identity \eqref{EQ:fundamental_eq} holds.
Then by using
\[
a_i=\prod_{p|M_i}p,\quad
b_i=\prod_{p|M_i}(p-1),\quad
a=a_1\cdots a_k,
\]
we can rewrite the identity \eqref{EQ:fundamental_eq} as
\[
\sum_{i=1}^{k}
\frac{b_i}{a_i}\prod_{p^e\parallel M_i}\left(1-\frac{1}{p^{e+1}}\right)^{-1}=1.
\]
Also, if we remove some prime factor of $M_1\cdots M_k$ from this identity,
then the left-hand side becomes smaller. Thus, we can apply Lemma \ref{Lem:HB_ineq2}
to obtain
\[
a\prod_{p^e\parallel M_1\cdots M_k}(p^{e+1}-1)\le F_{K}(a).
\]
Since $a=\Pi(\mathcal{P})$, this gives
\[
\sigma(M_1\cdots M_k)\Pi(\mathcal{P})\Phi(\mathcal{P})
\le
F_{K}(\Pi(\mathcal{P})).
\]
This completes the proof.
\end{proof}

\begin{proof}[Proof of Theorem \ref{Thm:main}]
If $\Pi(\mathcal{P})>2^{2^{K}}$, then we use Lemma \ref{Lem:carry_out_induction} to obtain
\[
\sigma(M_1\cdots M_k)\frac{\Phi(\mathcal{P})}{\Pi(\mathcal{P})}
\le
F_{2K}(2)\Pi(\mathcal{P})^{-2}
<
F_{2K}(2)2^{-2\cdot 2^{K}}.
\]
On the other hand, if $\Pi(\mathcal{P})\le2^{2^{K}}$, then we use Lemma \ref{Lem:Chen_Tang}
to obtain
\begin{align*}
\sigma(M_1\cdots M_k)\frac{\Phi(\mathcal{P})}{\Pi(\mathcal{P})}
&\le
F_K(\Pi(\mathcal{P}))
\Pi(\mathcal{P})^{-2}\\
&=
\Pi(\mathcal{P})^{2^{K-1}-2}\left(\Pi(\mathcal{P})^{2^{K-1}}-1\right)
\le
F_{2K}(2)2^{-2\cdot 2^{K}}.
\end{align*}
Thus in any case we have
\begin{equation}
\label{EQ:pre_final}
\sigma(M_1\cdots M_k)\frac{\Phi(\mathcal{P})}{\Pi(\mathcal{P})}\le F_{2K}(2)2^{-2\cdot 2^{K}}.
\end{equation}
Note that
\[
\sigma(M_1\cdots M_k)\frac{\Phi(\mathcal{P})}{\Pi(\mathcal{P})}
=
M_1\cdots M_k
\prod_{p^e\parallel M_1\cdots M_k}\left(1-\frac{1}{p^{e+1}}\right)
\]
Combining this identity with \eqref{EQ:pre_final},
we obtain
\[
M_1\cdots M_k
\le\prod_{p^e\parallel M_1\cdots M_k}\left(1-\frac{1}{p^{e+1}}\right)^{-1}F_{2K}(2)2^{-2\cdot 2^{K}}.
\]
By using
\[
\prod_{p^e\parallel M_1\cdots M_k}\left(1-\frac{1}{p^{e+1}}\right)^{-1}
<
\prod_{p}\left(1-\frac{1}{p^{2}}\right)^{-1}
=
\frac{\pi^2}{6},\quad
F_{2K}(2)<2^{4^{K}},
\]
we finally arrive at
\[
M_1\cdots M_k<\frac{\pi^2}{6}\,F_{2K}(2)2^{-2\cdot2^K}<\frac{\pi^2}{6}\,2^{4^{K}-2\cdot2^K}.
\]
This completes the proof.
\end{proof}

\begin{center}
\textbf{Acknowledgements.}
\end{center}

The author would like to express his gratitude to Prof.~Kohji Matsumoto
for his useful comments and encouragement. In particular, the discussion
with him enables the author to find a gap in the original manuscript
and to generalize the result to amicable tuples.
The author also would like to thank Mr.~Yuki Yoshida
for kindly providing his program for searching harmonious pairs
and giving useful advices on the searching program.
This work was supported by Grant-in-Aid for JSPS Research Fellow
(Grant Number: JP16J00906).

\bibliographystyle{amsplain}
\bibliography{/Users/suzuyu/Documents/Math/Tex/BibTeX/Perfect_Number,yos}

\vspace{3mm}
\begin{flushleft}
{\small
{\sc
Graduate School of Mathematics, Nagoya University,\\
Chikusa-ku, Nagoya 464-8602, Japan.
}

{\it E-mail address}: {\tt m14021y@math.nagoya-u.ac.jp}
}
\end{flushleft}
\end{document}